\documentclass[reqno,11pt]{amsart}
 
\usepackage{amsmath}
\usepackage{amsthm,amssymb,color,comment}
\usepackage{bbm}
\usepackage{mathrsfs}
\usepackage{hyperref}
\usepackage[TS1,T1]{fontenc}
\usepackage[utf8]{inputenc}
\usepackage{dsfont}
\usepackage{tikz}
\usepackage{enumitem}
\usepackage{subfigure}
\usepackage{mathtools}
\usepackage{bbold}
\usepackage{esint}
\usepackage[sort]{cite}% to arrange citations in increasing order

\numberwithin{equation}{section}

\newtheorem{thm}{Theorem}[section]

\newtheorem{lem}[thm]{Lemma}

\newtheorem{Def}[thm]{Definition}

\theoremstyle{definition}
\newtheorem{Ass}[thm]{Assumption}

\DeclareMathOperator*{\esssup}{ess\,sup}

\DeclareMathOperator{\DIV}{div}

\newcommand{\BL}{\text{BL}}
\newcommand{\KL}{\text{KL}}
\newcommand{\Lip}{\text{Lip}}

\newcommand{\R}{\mathbb{R}}

\newcommand{\Rd}{\mathbb{R}^{d}}

\newcommand{\eps}{\varepsilon}

\newcommand{\diff}{\mathop{}\!\mathrm{d}}

\makeatletter
\newcommand{\doublewidetilde}[1]{{%
  \mathpalette\double@widetilde{#1}%
}}
\newcommand{\double@widetilde}[2]{%
  \sbox\z@{$\m@th#1\widetilde{#2}$}%
  \ht\z@=.9\ht\z@
  \widetilde{\box\z@}%
}
\makeatother
\author{José A. Carrillo}
\address{{\it José A. Carrillo:} Mathematical Institute, University of Oxford, Woodstock Road, Oxford, OX2 6GG, United Kingdom}
\email{carrillo@maths.ox.ac.uk}

\author{Jakub Skrzeczkowski}
\address{{\it Jakub Skrzeczkowski: } Mathematical Institute, University of Oxford, Woodstock Road, Oxford, OX2 6GG, United Kingdom}
\email{jakub.skrzeczkowski@maths.ox.ac.uk}

\begin{document}

\title[]{Convergence and stability results for the particle system in the Stein gradient descent method}

\begin{abstract}
There has been recently a lot of interest in the analysis of the Stein gradient descent method, a deterministic sampling algorithm. It is based on a particle system moving along the gradient flow of the Kullback-Leibler divergence towards the asymptotic state corresponding to the desired distribution. Mathematically, the method can be formulated as a joint limit of time $t$ and number of particles $N$ going to infinity. We first observe that the recent work of Lu, Lu and Nolen (2019) implies that if $t \approx \log \log N$, then the joint limit can be rigorously justified in the Wasserstein distance. Not satisfied with this time scale, we explore what happens for larger times by investigating the stability of the method: if the particles are initially close to the asymptotic state (with distance $\approx 1/N$), how long will they remain close? We prove that this happens in algebraic time scales $t \approx \sqrt{N}$ which is significantly better. The exploited method, developed by Caglioti and Rousset for the Vlasov equation, is based on finding a functional invariant for the linearized equation. This allows to eliminate linear terms and arrive at an improved Gr{\"o}nwall-type estimate.  
\end{abstract}



\maketitle
\setcounter{tocdepth}{1}

\section{Introduction}

The Stein gradient descent method is a recently extensively studied algorithm \cite{MR4622922,han2018stein, liu2018stein,
wang2019stein, zhuo2018message,
chen2020projected,MR4582478,liu2016stein,liu2017stein,korba2020non,MR3919409} to sample the probability distribution $\rho_{\infty}:=e^{-V(x)}/Z$ when the normalization constant $Z = \int_{\Rd} e^{-V(x)} \diff x$ is unknown or difficult to compute. A prominent example is the Bayesian inference \cite{MR2652785} used to fit parameters $\theta \in \Theta$ based on the data $D$ and a priori distribution of parameters $\pi(\theta)$: the a posteriori distribution is given by
\begin{equation}\label{eq:bayes_formula}
\mathbb{P}(\theta|D) = \frac{\mathbb{P}(D|\theta) \, \pi(\theta)} {\int_{\Theta} \mathbb{P}(D|\theta') \, \pi(\theta') \diff \theta'} 
\end{equation}
Comparing to the well-known stochastic Metropolis-Hastings algorithm and its variants \cite{MR3363437,metropolis1953equation,MR3020951,MR1399158, MR1440273} which require huge number of iterations, the Stein algorithm is completely deterministic. In this method, one starts with a measure $\mu$ and modifies it via the map
$$
T_{\varepsilon,\phi}(x) = x +\varepsilon \, \phi,
$$
where $\varepsilon$ is a small parameter and $\phi$ is chosen to minimize the Kullback-Leibler divergence $\KL(T^{\#}_{\varepsilon, \phi} \, \mu || \rho_{\infty})$ where for two nonnegative measures $\mu$, $\nu$ the Kullback-Leibler divergence is defined as
\begin{equation}\label{eq:KL_div}
\KL(\mu || \nu) = \begin{cases}
\displaystyle\int_{\Rd} \log\left(\frac{\diff \mu}{\diff \nu}(x) \right) \, \frac{\diff \mu}{\diff \nu}(x)  \, \diff \nu(x)  &\mbox{ if } \frac{\diff \mu}{\diff \nu} \mbox{ exists,}\\
+\infty &\mbox{ otherwise, }
\end{cases} 
\end{equation}
and $T^{\#}_{\varepsilon, \phi}\mu$ is a measure which is the push-forward of $\mu$ along the map $T_{\varepsilon, \phi}$:
$$
T^{\#}_{\varepsilon, \phi}\mu(A) = \mu(T_{\varepsilon, \phi}^{-1}(A)) \qquad \qquad \mbox{ for any Borel set } A.
$$
The unique minimizer of the Kullback-Leibler divergence corresponds to the desired distribution $\rho_{\infty}$. The reason for choosing this functional is that its first variation does not depend on the normalization constant $Z$ (furthermore, this is the only functional with such a property, see \cite[Proposition 2.1]{chen2023gradient}). More precisely, $\phi$ is chosen as a maximizer of the following optimization problem
$$
\max_{\phi \in \mathcal{H}} \left\{ -\frac{\diff}{\diff \eps}\KL(T^{\#}_{\varepsilon, \phi} \, \mu || \rho_{\infty})|_{\varepsilon=0} : \|\phi\|_{\mathcal{H}} \leq 1 \right\}
$$
where $\mathcal{H}$ is a sufficiently big Hilbert space. A simple computation (see \cite{liu2016stein,liu2017stein}) shows that
$$
-\frac{\diff}{\diff \eps}\KL(T^{\#}_{\varepsilon, \phi} \, \mu || \rho_{\infty})|_{\varepsilon=0}  = \int_{\R^d} \left( \nabla \log(\rho_{\infty}) \cdot \phi + \DIV \phi \right) \diff \mu(x) 
$$
so that we see that the variation does not depend on the normalization constant $Z$. In the particular case that $\mathcal{H}$ is a reproducing Hilbert space with kernel $K(x-y)$, one can obtain an explicit expression for the optimal $\phi$ (up to a normalization constant):
$$
\phi \propto (\nabla \log(\rho_{\infty}) \mu )\ast K - \nabla K \ast \mu, 
$$
where $\ast$ denotes convolution operator $f\ast g(x) = \int_{\R^d} f(x-y)\,g(y) \diff y$. In particular, if $\mu$ has a particle representation, this motivates (formally) an iterative algorithm: we set $\mu_0 = \frac{1}{N}\sum_{i=1}^N \delta_{x_0^i}$ and given $\mu_l =\frac{1}{N}\sum_{i=1}^N \delta_{x_l^i}$ from the $l$-th step, in the $(l+1)$-th step we compute $\mu_{l+1} = \frac{1}{N}\sum_{i=1}^N \delta_{x_{l+1}^i}$ by
\begin{equation}\label{eq:discrete_algo_xi}
x^i_{l+1} = x_l^i + \frac{\varepsilon}{N}\sum_{j=1}^N \left[ \nabla \log \rho_{\infty}(x^j_l)\, K(x^i_l - x^j_l) - \nabla K(x^i_l - x^j_l)  \right] 
\end{equation}
(see \cite{liu2016stein,liu2017stein} for more details). This shows that the Stein gradient descent method is simple and attractive for practioners.\\

From the analytical point of view, moving from discrete distributions to the continuous ones (i.e. sending $N\to\infty$) is a delicate matter. Indeed, the Kullback-Leibler divergence \eqref{eq:KL_div} is not well-defined for the discrete distributions. However, its first variation is which makes the algorithm \eqref{eq:discrete_algo_xi} well-defined. In \cite{MR3919409}, the Stein's method was connected to the ODE system 
\begin{equation}\label{eq:ODE_mean_field_limit}
\partial_t x^i(t) = -\frac{1}{N}\sum_{j=1}^N \nabla K(x^i(t)-x^j(t))  - \frac{1}{N}\sum_{j=1}^N  K(x^i(t)-x^j(t)) \, \nabla V(x^j(t)).
\end{equation}
We note that the algorithm \eqref{eq:discrete_algo_xi} is in fact the time discretization of the ODE \eqref{eq:ODE_mean_field_limit} with the time step $\varepsilon$. Considering the empirical measure $\rho^N_t= \frac{1}{N} \sum_{i=1}^N \delta_{x_i(t)}$, it was proved in \cite{MR3919409}, that on finite intervals of time, $\rho^N_t \to \rho_t$ in the Wasserstein distance $\mathcal{W}_p$ where $\rho_t$ solves the nonlocal PDE
\begin{equation}\label{eq:nonlocal_PDE_Stein}
\partial_t \rho_t = \DIV(\rho_t \, K\ast(\nabla \rho_t + \nabla V \ \rho_t)).
\end{equation}
More rigorously, by using Dobrushin-type argument, the authors in \cite{MR3919409} established the following stability inequality
\begin{equation}\label{eq:stability_est_Wasserstein}
\mathcal{W}_p(\mu_t, \nu_t) \leq C\exp(C\exp(CT)) \mathcal{W}_p(\mu_0, \nu_0)
\end{equation}
for all times $t \in [0,T]$ and measure solutions $\mu_t$, $\nu_t$ to \eqref{eq:nonlocal_PDE_Stein}, assuming that $V(x) \approx |x|^p$ for large $x$ (see \cite{MR3919409} for more general setting). Having sent $N \to \infty$, one can obtain $\rho_{\infty}$ as the unique stationary solution of \eqref{eq:nonlocal_PDE_Stein} by sending $t\to \infty$.\\

\subsection{Main results} The paper \cite{MR3919409} recasts the Stein method as a limit $N \to \infty$ and then $t \to \infty$. Yet, practical computations involve discretization in space and so, they correspond in fact to the joint limit $N \to \infty$, $t\to \infty$. We first state a result showing that in a certain scaling between $N$ and $t$, one can rigorously justify the joint limit in the Wasserstein distance $\mathcal{W}_q$. This applies to the potentials having growth $|x|^p$ for large $x$.

\begin{thm}[convergence]\label{thm:logarithmic_estimates_Stein}
Suppose that $K$, $V$ satisfy Assumptions \ref{ass:general_for_both} and \ref{ass:theorem_lulunolen}. Let $\rho^N_t = \frac{1}{N} \sum_{i=1}^N \delta_{x_i(t)}$ where $x_i(t)$ solve \eqref{eq:ODE_mean_field_limit}. Let $N(t)= \exp(2\,C \exp(Ct))$ where $C$ is the constant as in \eqref{eq:stability_est_Wasserstein}. Then, for all $q \in [1,p)$
$$
\mathcal{W}_q(\rho_t^{N(t)}, \rho_{\infty}) \to 0 \mbox{ as } t\to \infty.
$$
\end{thm}

We see that the number of particles is unpractically large comparing to the time. To understand what happens for longer time scales, we address the question of stability of the particle system \eqref{eq:ODE_mean_field_limit}. Assuming that the initial configuration of particles $\rho_0^N$ is close to the asymptotic state (say, with error $\approx \frac{1}{N}$), we ask for how long it remains close. For example, the estimate \eqref{eq:stability_est_Wasserstein} suggests that after time $t \approx \log \log N$, the distance $\mathcal{W}_p(\rho^N_t, \rho_{\infty})$ is of order~1. Our main result improves this estimate and states that this time is of an algebraic order with respect to $N$ rather than just logarithmic.
\begin{thm}[stability estimate]\label{thm:main1}
Suppose that $K$, $V$ satisfy Assumptions \ref{ass:general_for_both} and \ref{ass:potential_kernel}. Let $\rho_t$ be a measure solution to \eqref{eq:nonlocal_PDE_Stein}. Then, there exists a constant $C$ depending only on $K$ and $V$ such that for all times $t$ satisfying $1-C\,t\,(t+1)\,\| \rho_0 - \rho_{\infty}\|_{\BL^*_V}>0$ we have
\begin{equation}\label{eq:stability_estimate_main_thm}
\| \rho_t - \rho_{\infty}\|_{\BL^*_V} \leq \frac{C\,(t+1) \, \| \rho_0 - \rho_{\infty}\|_{\BL^*_V}}{1-C\,(t+1)\,t\,\| \rho_0 - \rho_{\infty}\|_{\BL^*_V}}
\end{equation}
where the norm $\| \cdot \|_{\BL^*_V}$ is defined in \eqref{eq:weighted_norm}.

\end{thm}

Several comments are in order. First, conditions on $K$ and $V$ in Assumption \ref{ass:potential_kernel} are quite technical but they allow to consider all smooth, positive definite, sufficiently fast decaying kernels $K$ and potentials $V$ which grow at most like $|x|^2$ for large $x$. Second, the exploited distance distance $\|\cdot\|_{\BL^*_V}$ is a weighted modification of the bounded Lipschitz distance (also flat norm, Fortet-Mourier distance), commonly used in the analysis of transport-type PDEs (see, for instance, \cite{MR4309603} and Section \ref{sect:measures} for rigorous definition and related background).  Third, we see that when $\| \rho^N_0 - \rho_{\infty}\|_{\BL^*_V} \leq \frac{1}{N}$, then even for algebraic (with respect to $N$) time $t \leq \left(\frac{N}{2\,C}\right)^{1/2}-1$ we have
\begin{equation}\label{eq:ass_initial_configuration_nice}
1 - C\,t\,(t+1) \| \rho^N_0 - \rho_{\infty}\|_{\BL^*_V} \geq \frac{1}{2}
\end{equation}
so that with $\widetilde{C} := \left(\frac{2}{C}\right)^{1/2}$ we have
$$
\| \rho^N_t - \rho_{\infty}\|_{\BL^*_V} \leq \frac{\widetilde{C}}{\sqrt{N}},  \qquad \qquad  0 \leq t \leq \left(\frac{N}{2\,C}\right)^{1/2} - 1, 
$$
and so, possible instabilities in the particle system may occur much later compared to the time determined by the estimate \eqref{eq:stability_est_Wasserstein}.\\ 

The inspiration for Theorem \ref{thm:main1} comes from an insightful work of Caglioti and Rousset \cite{MR2358804,MR2448326} who obtained similar estimates for the Vlasov equation and the vortex method for the 2D Euler equation. The starting point is to consider the dual equation (which is common in the theory of solutions in the space of measures to transport-type PDEs, see for instance the monograph \cite{MR4309603}). In our case, we let 
$$
\mu_t := \rho_t - \rho_{\infty}.
$$
Since $\nabla \rho_{\infty} + \nabla V \, \rho_{\infty} = 0$, we have
$$
\partial_t \mu_t = \DIV(\mu_t \, K\ast(\nabla \mu_t + \nabla V \ \mu_t)) + \DIV(\rho_{\infty}\, K\ast(\nabla \mu_t + \nabla V \ \mu_t)).
$$
Let $g = g(T,x)$ be a smooth test function and consider the following dual equation
\begin{equation}\label{eq:dual_PDE}
\begin{split}
\partial_t \varphi = &\nabla \varphi \cdot K \ast (\nabla \mu_t + \mu_t \, \nabla V) + (\nabla \rho_{\infty}\, \varphi) \ast \nabla K - (\nabla \rho_{\infty} \, \varphi)\ast K \cdot \nabla V \\ &- (\rho_{\infty}\,\varphi)\ast \Delta K + (\rho_{\infty}\,\varphi)\ast\nabla K \cdot \nabla V + g\,(1+V),
\end{split}
\end{equation}
equipped with the terminal condition $\varphi(T,x) = 0$ for all $x \in \Rd$. Note carefully that $\varphi$ depends on the test function $g$ and the final time $T>0$. An easy computation shows that 
\begin{equation}\label{eq:duality_estimate_on_measure_solution}
\int_0^T \int_{\R^d} {g}(t,x)\, (1+V(x))\diff \mu_t(x) \diff t  = \int_{\R^d} \frac{\varphi(0,x)}{1+V(x)} \, (1+V(x))\diff \mu_0(x)
\end{equation}
so that to estimate $\mu_t$ on $[0,T]$, one needs to control $\frac{\varphi(0,x)}{1+V(x)}$, uniformly with respect to $g$. The crucial part of the argument in \cite{MR2448326} is to find a functional of the form $\mathcal{Q}(\varphi) \approx \int_{\Rd} w(x)\,|\varphi(x)|^2 \diff x$ (so that it is equivalent to a weighted $L^2$ norm of $\varphi$), which is invariant under the flow of linearization of \eqref{eq:dual_PDE}. As the time derivative of $\mathcal{Q}(\varphi)$ vanish on the linear terms of the dual equation, the estimate on $\mathcal{Q}(\varphi)$ will not yield exponential factors as obtained in \eqref{eq:stability_est_Wasserstein}. For Vlasov and Euler equations, the right choice was $w(x) = |\rho_{\infty}'(|x|)|$. In our case, we choose
$$
\mathcal{Q}(\varphi) = \int_{\Rd} \rho_{\infty}(x)\,|\varphi(x)|^2 \diff x.
$$
While this functional is not necessarily invariant, we prove that there is no positive contribution to its value under the flow of linearization of \eqref{eq:dual_PDE} (see Lemma \ref{lem:funct_Q}). This yields: 
\begin{thm}[weighted estimate on $\varphi$]\label{thm:weighted_estimate}
Suppose that $K$, $V$ satisfy Assumptions \ref{ass:general_for_both} and~\ref{ass:potential_kernel}. Let $\varphi$ be a solution to \eqref{eq:dual_PDE} with $g$ and $T>0$ fixed. Then, there exists a constant $C$ depending only on $V$ and $K$ such that
$$
\mathcal{Q}(\varphi(t,\cdot)) \leq C\,\int_t^T \|g(s,\cdot)\|_{L^{\infty}(\Rd)} \diff s \, e^{C\,\int_t^T \|\mu_s\|_{\BL^*_V} \diff s}.
$$
\end{thm}

With Theorem \ref{thm:weighted_estimate}, the proof of Theorem \ref{thm:main1} is a simple analysis of the explicit formula for solutions to \eqref{eq:dual_PDE} together with the Gr{\"o}nwall-type inequality, see Lemma \ref{lem:gronwall_exponential_ricatti}.\\

We remark that a non-rigorous reason why the functional $\mathcal{Q}$ is important in the analysis of linearized version of \eqref{eq:dual_PDE} is that its dual can be interpreted as a linearization of Kullback-Leibler divergence \eqref{eq:KL_div} around $\rho_{\infty}$. Indeed, writing $\rho = \rho_{\infty}+h$ where $\int_{\R^d} h = 0$ (to preserve the mass), we have
\begin{equation}\label{eq:linearization_KL}
\int_{\R^d} \rho \log\left(\frac{\rho}{\rho_{\infty}} \right) \diff x %= \int_{\R^d} (\rho_{\infty}+h) \, \log\left(1+\frac{h}{\rho_{\infty}} \right) 
\approx \int_{\R^d}\left( h + \frac{h^2}{\rho_{\infty}}\right) \diff x = \int_{\R^d} \frac{h^2}{\rho_{\infty}} \diff x.
\end{equation} 
One can wonder how to initially approximate the measure $\rho_{\infty}$ so that the condition \eqref{eq:ass_initial_configuration_nice} is satisfied. According to \cite[Theorem 4]{MR2448326}, almost each initial configuration satisfies a condition of this type. More precisely, if we restrict to dimension $d=2$ for simplicity. Let $\lambda_{\infty}$ be the product Lebesgue measure on $(\Rd)^{\infty} := \Rd \times \Rd \times ...$ (countably many times). Then, from \cite[Theorem 4]{MR2448326} we know that for all $\alpha \in (0,1/2)$, there exists a constant $C>0$ such that for $\lambda_{\infty}$-a.e. $\textbf{x} = (x_1, x_2, ...)$ the empirical measure $\rho^N[\textbf{x}] = \frac{1}{N} \sum_{i=1}^N \delta_{x_i}$ satisfy the estimate
$$
\| \rho^N[\textbf{x}] - \rho_{\infty}\|_{\BL^*_{V}} \leq \frac{C}{N^{\alpha}}. 
$$
Note that this statement exclude a lot of initial configurations in fact. For instance, if $A \subset \Rd$ is a set of measure zero, then $\Rd \times \Rd \times A \times \Rd \times \Rd \times ...$ is of measure zero in $(\Rd)^{\infty}$. Similar results are valid in arbitrary dimensions but the space $\BL^*_V$ has to be slightly modified, see \cite[Theorem 4]{MR2448326}.\\  

Of course, one would like to see estimates which shows that the gradient flow improves the initial estimate rather than just does not worsen it too much. This is a difficult problem, far beyond the scope of the current manuscript.\\

Let us comment the novelties of the manuscript and put them in the context of other works. From the analytical point of view, estimates of the form \eqref{eq:stability_estimate_main_thm} have been only obtained before by Caglioti and Rousset for the Vlasov equation and for the vorticity formulation of 2D Euler equation \cite{MR2448326, MR2358804}. The method uses the functional $\mathcal{Q}$ which is conservative for the flow of the linearized dual problem and is constructed by the methods of Hamiltonian mechanics. In our case, the functional $\mathcal{Q}$ has different form and can be rather interpreted via linearization of Kullback-Leibler divergence, see \eqref{eq:linearization_KL}. More generally, our work shows that the idea of \cite{MR2448326, MR2358804} can be possibly applied to a much broader class of PDEs without a Hamiltonian structure. Concerning the particular case of the Vlasov equation, we also mention the work of Han-Kwan and Nguyen \cite{MR3463791} who proved a negative result: if $f_{\infty} = f_{\infty}(v)$ is an unstable equilibrium (in the sense of so-called Penrose instability condition) and initially $\mathcal{W}_1(\mu^N_0, f_{\infty}) \approx \frac{1}{N^{\alpha}}$ for $\alpha>0$ sufficiently small then 
$
\limsup_{N \to \infty} \mathcal{W}_1(\mu^N_{T_N}, f_{\infty})>0
$
for $T_N = O(\log N)$.\\  

From the point of view of numerical analysis and statistics, the only available estimate addressing the convergence of the particle system in the Stein method is \eqref{eq:stability_est_Wasserstein} obtained by Lu, Lu and Nolen in \cite{MR3919409} which belongs to the large class of convergence results of mean-field limits for Vlasov equation and aggregation equation \cite{MR0541637, MR0475547, MR2278413, MR3377068, MR3558251, MR4158670}. First, from their result we deduced convergence of the method assuming that $t \approx \log\log N$, see Theorem \ref{thm:logarithmic_estimates_Stein}. Moreover, we provided stability estimates for the longer, more practical timescale $t \approx \sqrt{N}$ which, up to our knowledge, are entirely new. Other interesting results for the Stein method focus on the convergence of  \eqref{eq:discrete_algo_xi}, assuming that the time step $\varepsilon$ is sufficiently small and the initial distribution is an absolutely continuous measure \cite{liu2017stein,korba2020non}, or the analysis of the asymptotics $t\to\infty$ via log-Sobolev-type inequalities \cite{MR4582478}, which is still a programme far from being complete. In both of these approaches, one needs the continuity of the initial distribution to consider the Kullback-Leibler divergence \eqref{eq:KL_div} which is not well-defined for discrete measures.\\

The paper is structured as follows. In Section~\ref{sect:measures}, we review the theory of spaces of measures and we define the norm $\|\cdot\|_{\BL^*_V}$. We also define measure solutions to \eqref{eq:nonlocal_PDE_Stein}. In Section \ref{sect:assumption} we introduce assumptions on the potential $V$ and the kernel $K$. In Section \ref{sect:proof_loglog_est} we prove Theorem \ref{thm:logarithmic_estimates_Stein} while in Section \ref{sect:weighted_estimate} we prove Theorem \ref{thm:weighted_estimate} which allows to demonstrate Theorem \ref{thm:main1} in Section \ref{sect:BL_est}.

\section{Measure solutions and the functional analytic setting}\label{sect:measures}

\subsection{Spaces of measures, the Wasserstein distance and the weighted bounded Lipschitz distance} We first introduce the functional analytic framework. The most common notion of distance in the space of measures is probably the Wasserstein distance
\begin{equation}\label{eq:Wasserstein_distance}
\mathcal{W}_p(\mu, \nu) = \inf_{\pi \in \Pi(\mu,\nu)} \left(\int_{\R^d \times \R^d} |x-y|^p \diff \pi(x,y)\right)^{1/p}
\end{equation}
where $\Pi(\mu, \nu)$ is a set of couplings between $\mu$ and $\nu$, i.e. $\pi \in \Pi(\mu, \nu)$ if $\pi$ is a probability measure on $\R^d\times\R^d$ such that $\pi(A \times \R^d) = \mu(A)$ and $\pi(\R^d \times B) = \nu(B)$. Definition \eqref{eq:Wasserstein_distance} requires that $\int_{\R^d} (1+|x|^p) \diff \mu, \int_{\R^d} (1+|x|^p) \diff \nu < \infty$. Theorem \ref{thm:logarithmic_estimates_Stein} is formulated using the Wasserstein distance.\\ 

For Theorem \ref{thm:main1}, we need a notion of distance compatible with the duality method. This will be the bounded Lipschitz distance. To define it, we first introduce the space of bounded Lipschitz functions
$$
\BL(\R^d) = \{\psi:\R^d \to \R: \|\psi\|_{L^{\infty}(\R^d)} < \infty, \, |\psi|_{\Lip} < \infty\},
$$
where
$$
|\psi|_{\Lip} = \sup_{x\neq y} \frac{|\psi(x)-\psi(y)|}{|x-y|}.
$$
A useful fact is that $|\psi|_{\Lip} \leq \|\nabla \psi\|_{L^{\infty}(\R^d)}$. In particular, to estimate $\|\psi\|_{\BL(\R^d)}$, it is sufficient to compute $\|\psi\|_{L^{\infty}(\R^d)}$ and $\|\nabla \psi\|_{L^{\infty}(\R^d)}$.\\ 

Given an arbitrary signed measure $\mu \in \mathcal{M}(\R^d)$, we recall its unique Hahn-Jordan decomposition $\mu = \mu^+ - \mu^-$ where both measures $\mu^+$, $\mu^-$ are nonnegative. We define total variation of $\mu$ as a nonnegative measure $|\mu|:= \mu^+ + \mu^-$. Then, for any signed measure $\mu$ such that $\int_{\R^d}(1+V(x))\diff |\mu|(x) <\infty$, we define its weighted bounded Lipschitz norm as  
\begin{equation}\label{eq:weighted_norm}
\|\mu\|_{\BL^*_V} := \sup_{\|\varphi\|_{\BL} \leq 1} \int_{\R^d} \varphi(x)\, (1+V(x)) \diff \mu(x).
\end{equation}
This is a weighted variant of the bounded Lipschitz norm
\begin{equation}\label{eq:flat_noweighted_norm}
\|\mu\|_{\BL^*} := \sup_{\|\varphi\|_{\BL} \leq 1} \int_{\R^d} \varphi(x) \diff \mu(x),
\end{equation}
widely used in the analysis of PDEs, when the total mass is not conserved \cite{MR3870087,MR2536440,MR4347325} (otherwise, one can use the Wasserstein distance). In our case, we introduce an additional weight $V(x)$ to address the growth $V(x) \to \infty$ as $|x|\to \infty$. Similar weighted norms were introduced before to remove singularities in the studied problems \cite{MR4565590,szymanska2021bayesian}.  

\subsection{Measure solutions to \eqref{eq:nonlocal_PDE_Stein}} 
The measure solution is defined as follows.
\begin{Def}[measure solution]\label{def:measure_solution}
We say that a family of probability measures $\{\rho_t\}_{t\in[0,T]}$ is a measure solution to \eqref{eq:nonlocal_PDE_Stein} if $t \mapsto \rho_t$ is continuous (with respect to the narrow topology), for all $T>0$ the growth estimate
\begin{equation}\label{eq:estimate_C(T)_on_finite_time_intervals}
\sup_{t \in [0,T]} \|\rho_t\|_{\BL^*_V} = \sup_{t \in [0,T]} \int_{\R^d} (1+V(x)) \diff \rho_t(x) \leq C(T) 
\end{equation}
is satisfied and for all $\phi \in C_c^{\infty}([0,\infty)\times\R^d)$
$$
\int_0^{\infty} \int_{\R^d} \partial_t \phi(t,x) + \nabla \phi(t,x) \cdot K\ast(\nabla \rho_t + \rho_t\,\nabla V) \diff \rho_t(x) \diff t + \int_{\R^d} \phi(0,x) \diff \rho_0(x) = 0.
$$
\end{Def}

Three explanations are in order. First, the continuity with respect to the narrow topology means that the map $t\mapsto \int_{\R^d} \psi(x) \diff \rho_t(x)$ is continuous for all $\psi:\R^d \to \R$ bounded and continuous. Second, the first equality in \eqref{eq:estimate_C(T)_on_finite_time_intervals} follows by nonnegativity of $\rho_t$. Third, it is a simple computation to see that $\rho_t^N = \frac{1}{N} \sum_{i=1}^N \delta_{x_i(t)}$, where $x_i(t)$ solves \eqref{eq:ODE_mean_field_limit} with initial condition $x_i(0)$, is a measure solution to \eqref{eq:nonlocal_PDE_Stein} with initial condition $\rho_0^N = \frac{1}{N} \sum_{i=1}^N \delta_{x_i(0)}$.\\

The measure solution to \eqref{eq:nonlocal_PDE_Stein} as in Definition \ref{def:measure_solution} was constructed in \cite{MR3919409}. In our work, we will need (stronger) continuity in time with respect to the $\|\cdot\|_{\BL^*_V}$ norm which is given by the following lemma.

\begin{lem}\label{lem:continuity_in_time} Let $\{\rho_t\}$ be a measure solution to \eqref{eq:nonlocal_PDE_Stein}. Then, for all $T>0$, there exists a constant $C$ depending on $\rho_0$, $V$, $K$ and $C(T)$ in \eqref{eq:estimate_C(T)_on_finite_time_intervals} such that for all $s,t \in [0,T]$
$$
\|\rho_t - \rho_s \|_{\BL^*_V} \leq C\, |t-s|.
$$
\end{lem}
As the computation is classical, we provide a short proof in the Appendix \ref{app:cont_in_time}.

\section{Assumptions on the kernel and the potential}\label{sect:assumption}
For the sake of clarity, we specify assumptions for Theorems \ref{thm:logarithmic_estimates_Stein} and \ref{thm:main1} separately.  
\begin{Ass}\label{ass:general_for_both} For both Theorems \ref{thm:logarithmic_estimates_Stein} and \ref{thm:main1} we assume that:
\begin{itemize}
\item $K$ is nonnegative, symmetric $K(x)= K(-x)$ and positive-definite, i.e. for all test functions $\xi: \R^d \to \R^d$ we have $\int_{\R^d} K \ast \xi \cdot \xi \diff x \geq 0$,
\item $V$ is a smooth and nonnegative function,
\item there exists $p>0$, $C>0$ and $R>0$ such that for all $x$ with $|x|>R$ we have
\begin{equation}\label{eq:growth_V_simple}
\begin{split}
\frac{1}{C}\, (|x|^p - 1) \leq &\,V(x) \leq C\, (|x|^p + 1),\\
\frac{1}{C}\, (|x|^{p-1} - 1) \leq &\,|\nabla V(x)| \leq C\, (|x|^{p-1} + 1), \\
\frac{1}{C}\, (|x|^{p-2} - 1) \leq &\,|\nabla^2 V(x)| \leq C\, (|x|^{p-2} + 1). 
\end{split}
\end{equation}
\end{itemize}
\end{Ass}
\begin{Ass}\label{ass:theorem_lulunolen} For Theorem \ref{thm:logarithmic_estimates_Stein} we assume additionally (as in \cite{MR3919409}) that:
\begin{itemize}
\item condition \eqref{eq:growth_V_simple} holds with $p > 1$,
\item $K \in C^4(\R^d) \cap W^{4,\infty}(\R^d)$,
\item there exists smooth $K_{1/2}$ such that $K = K_{1/2} \ast K_{1/2}$ and its Fourier transform $\widehat{K_{1/2}}$ is positive.
\end{itemize}
\end{Ass}
\begin{Ass}\label{ass:potential_kernel}
For Theorem \ref{thm:main1} we assume additionally that: 
\begin{itemize}
    \item $K\in W^{3,\infty}(\R^d)$,
    \item $V$ and $K$ satisfy the following conditions: \begin{equation}\label{eq:add_cond_growth_at_infinity}     \sup_{x\in\R^d} \left\|  \frac{\nabla V(x)}{1+V(\cdot)}\cdot \nabla K(x-\cdot)\right\|_{\BL}, \, \sup_{x\in\R^d} \left\| \frac{\nabla V(x) \cdot \nabla V(\cdot)}{1+V(\cdot)}\, K(x-\cdot) \right\|_{\BL} < \infty. \end{equation}
\end{itemize} 
\end{Ass}
The only condition which is difficult to understand is \eqref{eq:add_cond_growth_at_infinity}. Unfortunately, it restricts our reasoning to the case of $V$ which can be at most quadratic at infinity (i.e. $p \leq 2$ in \eqref{eq:growth_V_simple}).
\begin{lem}\label{lem:example_ass_satisfied}
Suppose that $V, K \geq 0$, $V \in W^{2,\infty}_{\text{loc}}(\R^d)$, $K\in W^{1,\infty}(\R^d)$ and assume that $V$ satisfies \eqref{eq:growth_V_simple} with exponent $p$. Furthermore, suppose that $|\nabla V|\, K, |\nabla V|\, |\nabla K| \in L^{\infty}(\R^d)$. Then, $V$ and $K$ satisfy \eqref{eq:add_cond_growth_at_infinity} if and only if $p \in (0,2]$.
\end{lem}

The proof is presented in Appendix \ref{app:proofs_assumptions}.\\

We conclude with a crucial consequence of \eqref{eq:add_cond_growth_at_infinity} which provides bounds on the vector field for the transport equation \eqref{eq:dual_PDE}.
\begin{lem}\label{lem:est_vf_Rd}
Let $\mu$ be a measure such that $\|\mu\|_{\BL^*_V}<\infty$. Then, there exists a constant $C$ depending only on $V$ and $K$ such that
\begin{equation}\label{eq:estimate_vf_whole_space_1}
 \|K \ast (\nabla \mu + \mu \, \nabla V) \|_{L^{\infty}(\R^d)}, \, \,  \|\nabla K \ast (\nabla \mu + \mu \, \nabla V) \|_{L^{\infty}(\R^d)} \leq C\, \|\mu\|_{\BL^*_V}.
\end{equation}
Moreover,
\begin{equation}\label{eq:estimate_vf_whole_space_2}
\left\| \nabla V \cdot K \ast (\nabla \mu + \mu \, \nabla V) \right\|_{L^{\infty}(\Rd)} \leq C\, \|\mu\|_{\BL^*_V}.
\end{equation}
\end{lem}

The proof is presented in Appendix \ref{app:proofs_assumptions}.

\section{Proof of Theorem \ref{thm:logarithmic_estimates_Stein}}\label{sect:proof_loglog_est}
Here, we prove that in the particular scaling $t \approx \log\log N$, we can pass to the joint limit $t, N \to \infty$. The result is a simple consequence of results in \cite{MR3919409}. \\

We recall that in \cite{MR3919409} the Authors prove that 
\begin{equation}\label{eq:stability_estimate_Dobrushin}
\mathcal{W}_p(\rho_t^N, \rho_t) \leq C \exp(C \exp(Ct))\, \mathcal{W}_p(\rho_0^N, \rho_0) \leq \frac{C\exp(C \exp(Ct))}{N},
\end{equation}
where $C$ depends only on $V$ and $K$. Furthermore, if $\rho_0$ is an absolutely continuous measure, there exists a unique solution $\rho_t$ to \eqref{eq:nonlocal_PDE_Stein} which is an absolutely continuous measure for all $t \in [0,\infty)$ and
\begin{equation}\label{eq:convergence_measures_weak}
\rho_t \to \rho_{\infty} \mbox{ narrowly (as measures)}.
\end{equation}
The target of this Section is to combine \eqref{eq:stability_estimate_Dobrushin} and \eqref{eq:convergence_measures_weak} to prove the Theorem \ref{thm:logarithmic_estimates_Stein}. We will first upgrade the convergence \eqref{eq:convergence_measures_weak}.
\begin{lem}\label{lem:convergence_rho_t_Wasserstein}
Suppose that $\{\rho_t\}$ is an (absolutely contiuous) measure solution to \eqref{eq:nonlocal_PDE_Stein} with initial condition $\rho_0$ such that $\KL(\rho_0|\rho_{\infty}) < \infty$. Then, for all $1 \leq q < p$ we have $\mathcal{W}_q(\rho_t, \rho_{\infty}) \to 0$ when $t \to \infty$.
\end{lem}
\begin{proof}
As in \cite{MR3919409}, since we deal with the absolutely continuous solution $\rho_t$, we have inequality
$$
\partial_t \KL(\rho_t|\rho_\infty) + \int_{\R^d}  (\nabla \rho_t + \nabla V \rho_t)\,  K \ast (\nabla \rho_t + \nabla V \rho_t) \diff x   \leq 0
$$
so that $
\KL(\rho_t|\rho_{\infty}) \leq \KL(\rho_0|\rho_{\infty}) < \infty.
$
This means that
$$
\int_{\R^d} \rho_t \log(\rho_t) + V(x) \rho_t \diff x \leq C.
$$ 
By standard arguments (for instance, splitting the set $\{x \in \R^d: \rho_t \leq 1\}$ for two sets: $\{x \in \R^d : \rho_t \leq e^{-|x|^p/\sigma} \}$ and $\{x \in \R^d: e^{-|x|^p/\sigma} \leq \rho_t \leq 1\}$) we have
$$
-\int_{\R^d} \rho_t \log^-(\rho_t) \diff x \leq \int_{\R^d} e^{-|x|^p/(2\sigma)} \diff x + \sigma\, \int_{\R^d} \rho_t \, |x|^p \diff x
$$
where $\log^-$ is the negative part of $\log$. Hence, choosing $\sigma$ small enough, using that $\rho_t$ is a probability measure and growth conditions on $V$ in \eqref{eq:growth_V_simple}, we get
\begin{equation}\label{eq:bounds_solution_Xp}
\sup_{t\in[0,\infty)} \int_{\R^d} \left( \rho_t |\log(\rho_t)| + |x|^p \rho_t \right) \diff x \leq C.
\end{equation} 
This inequality gives tightness of the sequence $\{\rho_t\}$ in \cite{MR3919409} to prove \eqref{eq:convergence_measures_weak} but in our case, it gives us uniform moment estimate which is relevant in the sequel. \\

To prove the lemma, by \cite[Theorem 5.11]{MR3409718}, it is sufficient to prove $\int_{\R^d} |x|^q \rho_t(x) \diff x \to \int_{\R^d} |x|^q \rho_{\infty}(x) \diff x$. Let $T_R$ be the truncation operator defined as 
$$
T_R(y) = \begin{cases}
y &\mbox{ if } |y|\leq R,\\
\frac{y}{|y|}\,R &\mbox{ if } |y| > R,
\end{cases}
$$
so that $|T_R(y)| \leq |y|$ and $|T_R(y)| \leq R$. Hence,
\begin{equation}\label{eq:splitting_the_integral}
\begin{split}
&\left|\int_{\R^d} |x|^q (\rho_t - \rho_{\infty}) \diff x \right| \leq \\
& \qquad \qquad \qquad \leq 
\left|\int_{\R^d} T_R(|x|^q) (\rho_t - \rho_{\infty}) \diff x \right| + 
\left|\int_{\R^d} (|x|^q - T_R(|x|^q)) (\rho_t - \rho_{\infty}) \diff x \right|.
\end{split}
\end{equation}
In the second integral we can restrict to $|x|>R$ so that 
$$
||x|^q - T_R(|x|^q)| \leq 2\, |x|^p \, R^{q-p}.
$$
By \eqref{eq:bounds_solution_Xp} and $\int_{\R^d} |x|^p \rho_{\infty} \diff x < \infty$, we conclude
\begin{equation}\label{eq:tail_estimate_large_x}
\left|\int_{\R^d} (|x|^q - T_R(|x|^q)) (\rho_t - \rho_{\infty}) \diff x \right| \leq C\, R^{q-p}.
\end{equation}
As $T_R(|x|^q)$ is an admissible test function for the narrow convergence, we deduce from \eqref{eq:splitting_the_integral} and \eqref{eq:tail_estimate_large_x}
$$
\limsup_{t \to \infty} \left|\int_{\R^d} |x|^q (\rho_t - \rho_{\infty}) \diff x \right| \leq C\, R^{q-p}.
$$
As $R$ is arbitrary, the proof is concluded. 
\end{proof}
\begin{proof}[Proof of Theorem \ref{thm:logarithmic_estimates_Stein}] 
By the triangle inequality
$$
\mathcal{W}_q(\rho_t^N, \rho_\infty) \leq \mathcal{W}_q(\rho_t^N, \rho_t) + \mathcal{W}_q(\rho_t, \rho_{\infty}).
$$
In view of \eqref{eq:stability_estimate_Dobrushin} and a simple inequality $\mathcal{W}_q(\rho_t^N, \rho_t) \leq \mathcal{W}_p(\rho_t^N, \rho_t)$ (as $q \leq p$)
$$
\mathcal{W}_q(\rho_t^N, \rho_t) \leq \frac{C\exp(C \exp(Ct))}{N}.
$$ 
Hence, if $N= \exp(2\,C \exp(Ct))$ then $\mathcal{W}_q(\rho_t^N, \rho_t) \to 0$ and so, the conclusion follows by Lemma \ref{lem:convergence_rho_t_Wasserstein}.
\end{proof}

\section{The weighted $L^2$ estimate (proof of Theorem \ref{thm:weighted_estimate})}\label{sect:weighted_estimate}
%%%%%%%%%%%%%%%%%% TORUS COM
\iffalse 
\subsection{Estimates on the vector field (the case $\Rd$)}

\begin{lem}\label{lem:est_vf}
Let $\mu$ be a measure. Then, for all $k = 0, 1, ..., m$ we have
$$
 \|\nabla^k K \ast (\nabla \mu + \mu \, \nabla V) \|_{L^{\infty}(\Rd)} \leq C\, \|\mu\|_{\BL^*_m}.
$$
\end{lem}
\begin{proof}
Let $k = 0$. Then,
\begin{multline*}
\|K \ast \nabla \mu\|_{L^{\infty}(\Rd)} = \|\nabla K \ast \nabla \mu\|_{L^{\infty}(\Rd)}= \sup_{x\in\Rd} \left| \int_{\Rd} \nabla K(x-y) \diff \mu(y)\right| \leq \\ 
\leq 
\sup_{x\in\Rd} \| \nabla K(x-\cdot)\|_{\BL_m} \, \| \mu \|_{\BL^*_m} \leq 
\| \nabla K\|_{\BL_m}\, \| \mu \|_{\BL^*_m} \leq \| K\|_{\BL_{m+1}}\, \| \mu \|_{\BL^*_m}.
\end{multline*}
In the same way, we prove that $\|\nabla^k K \ast \nabla \mu\|_{L^{\infty}(\Rd)} \leq \| K\|_{\BL_{m+k}}\, \| \mu \|_{\BL^*_m}$. Concerning the term $\nabla^k K \ast (\mu\,\nabla V)$ we let again $k=0$ so that
\begin{multline*}
\|K \ast (\mu\,\nabla V) \|_{L^{\infty}(\Rd)} = \sup_{x\in\Rd} \left|\int_{\Rd} K(x-y)\,\nabla V(y)\diff \mu(y) \right| \leq \\ \leq 
\sup_{x\in\Rd} \| K(x-\cdot)\,\nabla V(\cdot)\|_{\BL_m}\, \| \mu \|_{\BL^*_m} \leq 
\|K\|_{\BL_m}\, \| V\|_{\BL_{m+1}} \, \| \mu \|_{\BL^*_m}. 
\end{multline*}
Similarly, we prove that $\|\nabla^k K \ast \nabla \mu\|_{L^{\infty}(\Rd)} \leq \|K\|_{\BL_{m+k}}\, \| V\|_{\BL_{m+1}} \, \| \mu \|_{\BL^*_m}$.
\end{proof}
\fi
%%%%%%%%%%%%%%%%%%%%%%%%%

We first present the crucial cancellation lemma which allows to cancel the terms which are linear with respect to $\varphi$. 
\begin{lem}\label{lem:funct_Q} Let $\mathcal{Q}(\varphi) := \left(\int_{\Rd} \rho_{\infty} \, |\varphi|^2 \diff x\right)^{1/2}$ and let
$$
f(\varphi) :=  (\nabla \rho_{\infty}\, \varphi) \ast \nabla K - (\nabla \rho_{\infty} \, \varphi)\ast K \cdot \nabla V  - (\rho_{\infty}\,\varphi)\ast \Delta K + (\rho_{\infty}\,\varphi)\ast\nabla K \cdot \nabla V.
$$
Then, 
\begin{equation}\label{eq:lower_bound_on_the_linear_contribution}
\int_{\Rd} f(\varphi)\,\varphi\, \rho_{\infty} \diff x = \int_{\Rd} (\nabla \varphi \, \rho_{\infty}) \ast K \cdot (\nabla \varphi \, \rho_{\infty}) \diff x \geq 0.
\end{equation} 
In particular, if $\varphi$ solves \eqref{eq:dual_PDE}, then
\begin{equation}\label{eq:time_der_Q}
\partial_t \mathcal{Q}(\varphi) \geq \frac{1}{\mathcal{Q}(\varphi)}\, \int_{\Rd} \rho_{\infty} \, \varphi\, g\, (1+V) \diff x + \frac{1}{\mathcal{Q}(\varphi)} \, \int_{\Rd} \rho_{\infty} \, \varphi\, \nabla \varphi \cdot K \ast (\nabla \mu_t + \mu_t \, \nabla V) \diff x.
\end{equation}
\end{lem}

\begin{proof}[Proof of Lemma \ref{lem:funct_Q}]
The most important observation is that $- \rho_{\infty} \nabla V = \nabla \rho_{\infty}$. Hence,
\begin{equation*}
\begin{split}
\int_{\Rd} f(\varphi)\,\varphi\, \rho_{\infty} \diff x = & 
\int_{\Rd} (\nabla \rho_{\infty}\, \varphi) \ast \nabla K \, \varphi \, \rho_{\infty} + (\nabla \rho_{\infty} \, \varphi)\ast K \cdot \nabla \rho_{\infty} \, \varphi \diff x \\  &- \int_{\Rd}  (\rho_{\infty}\,\varphi)\ast \Delta K \, \varphi \, \rho_{\infty} + (\rho_{\infty}\,\varphi)\ast\nabla K \cdot \nabla \rho_{\infty} \, \varphi \diff x =: I_1 + I_2. 
\end{split}
\end{equation*}
We observe that $I_1 = -\int_{\Rd} (\nabla \rho_{\infty}\, \varphi) \ast  K \cdot \nabla \varphi \, \rho_{\infty} \diff x$ by integrating by parts in the first term in $I_1$. Similarly $I_2 = \int_{\Rd}  (\rho_{\infty}\,\varphi)\ast \nabla K \cdot \nabla \varphi \, \rho_{\infty} \diff x$. Now, by standard properties of convolutions,  $(\rho_{\infty}\,\varphi)\ast \nabla K = (\nabla \rho_{\infty}\,\varphi)\ast K + ( \rho_{\infty}\,\nabla \varphi)\ast K$ so that summing $I_1 + I_2$ we conclude the proof of \eqref{eq:lower_bound_on_the_linear_contribution} (the nonnegativity follows by the positive definiteness of $K$). Concerning \eqref{eq:time_der_Q}, we observe that differentiating in time and using PDE \eqref{eq:dual_PDE}
\begin{multline*}
\partial_t \mathcal{Q}(\varphi) \, \mathcal{Q}(\varphi) = \\ =
\int_{\Rd} \rho_{\infty} \, \varphi\, f(\varphi) \diff x  +  \int_{\Rd} \rho_{\infty} \, \varphi\, g\, (1+V) \diff x + \int_{\Rd} \rho_{\infty} \, \varphi\, \nabla \varphi \cdot K \ast (\nabla \mu_t + \mu_t \, \nabla V) \diff x
\end{multline*}
so that \eqref{eq:time_der_Q} follows directly from \eqref{eq:lower_bound_on_the_linear_contribution}.
\end{proof}

\begin{proof}[Proof of Theorem \ref{thm:weighted_estimate}]
Integrating \eqref{eq:time_der_Q} in time and using $\varphi(T,x) = 0$ we deduce 
\begin{multline*}
\mathcal{Q}(\varphi(t,\cdot)) \leq  \int_t^T \frac{1}{\mathcal{Q}(\varphi(s,\cdot))} \left|\int_{\Rd} \rho_{\infty}(x) \, \varphi(s,x)\, g(s,x) \, (1+V(x)) \diff x \right| \diff s + \\ + \int_t^T \frac{1}{\mathcal{Q}(\varphi(s,\cdot))}  \left|\int_{\Rd} \rho_{\infty}(x) \, \varphi(s,x)\, \nabla \varphi(s,x) \cdot K \ast (\nabla \mu_s + \mu_s \, \nabla V) \diff x\right| \diff s =: J_1 + J_2.
\end{multline*}
\underline{Estimate on $J_1$.} We use H{\"o}lder inequality to obtain
\begin{multline*}
\frac{1}{\mathcal{Q}(\varphi(s,\cdot))} \left|\int_{\Rd} \rho_{\infty}(x) \, \varphi(s,x)\, g(s,x) \, (1+V(x)) \diff x \right| \leq \left(\int_{\Rd}
\rho_{\infty}(x) \, g^2(s,x) \, (1+V(x))^2 \diff x \right)^{1/2} \\
\leq \|\rho_{\infty}\,(1+V)^2\|_{L^1(\Rd)}^{1/2} \, \|g(s,\cdot)\|_{L^{\infty}(\Rd)} \leq C\,\|g(s,\cdot)\|_{L^{\infty}(\Rd)}
\end{multline*}
so that $J_1$ can be estimated by $C\,\int_t^T \|g(s,\cdot)\|_{L^{\infty}(\R^d)} \diff s$.\\

\underline{Estimate on $J_2$.} We write $
\varphi\,\nabla \varphi = \frac{1}{2}\nabla \varphi^2$ and integrate by parts to get two terms:
\begin{multline*}
\int_t^T \frac{1}{\mathcal{Q}(\varphi(s,\cdot))}  \left|\int_{\Rd} \nabla \rho_{\infty}(x) \, \varphi^2(s,x) \cdot K \ast (\nabla \mu_s + \mu_s \, \nabla V) \diff x\right| \diff s \, + \\
+ \int_t^T \frac{1}{\mathcal{Q}(\varphi(s,\cdot))}  \left|\int_{\Rd}  \rho_{\infty}(x) \, \varphi^2(s,x) \, \nabla K \ast (\nabla \mu_s + \mu_s \, \nabla V) \diff x\right| \diff s.
\end{multline*}
Using $\nabla \rho_{\infty} = -\rho_{\infty}\,\nabla V$, we can estimate it by
$$
\int_t^T {\mathcal{Q}(\varphi(s,\cdot))} \, \left( \left\| \nabla V \cdot K \ast (\nabla \mu_s + \mu_s \, \nabla V) \right\|_{L^{\infty}(\Rd)} + \left\| \nabla K \ast (\nabla \mu_s + \mu_s \, \nabla V) \right\|_{L^{\infty}(\Rd)} \right) \diff s.  
$$
The $L^{\infty}$ norms above can be bounded by $C\,\|\mu_s\|_{\BL^*_V}$ using Lemma \ref{lem:est_vf_Rd} so that we obtain
$$
\mathcal{Q}(\varphi(t,\cdot)) \leq C\, \int_t^T \|g(s,\cdot)\|_{ L^{\infty}(\Rd)} \diff s + C\, \int_t^T {\mathcal{Q}(\varphi(s,\cdot))} \, \|\mu_s\|_{\BL^*_V} \diff s. 
$$
Using Lemma \ref{lem:back_gronwall}, we conclude the proof. 
\end{proof}

\section{BL estimates on $\varphi$ and proof of Theorem \ref{thm:main1}}\label{sect:BL_est} The plan is to write explicit solution to \eqref{eq:dual_PDE} and estimate each term separately. Note that for a general transport equation 
$\partial_t \varphi = \nabla \varphi \cdot b(t,x) + c(t,x)$, $\varphi(T,x) = 0$,
the method of characteristics yields the following representation formula
$$
\varphi(t,x) = - \int_t^T c(s, X_{t,s}(x)) \diff s,
$$
where $X_{t,s}(x)$ is the flow of the vector field $b$:
$$
\partial_s X_{t,s}(x) = -b(s,X_{t,s}(x)), \qquad  X_{t,t}(x) = x.
$$
Therefore, the solution to \eqref{eq:dual_PDE} can be written as
\begin{equation}\label{eq:sol_dual_PDE}
\begin{split}
\varphi(t,x) = & -\int_t^T(\nabla \rho_{\infty}\, \varphi) \ast \nabla K(s,X_{t,s}(x)) \diff s + \int_t^T (\nabla \rho_{\infty} \, \varphi)\ast K \cdot \nabla V(s,X_{t,s}(x)) \diff s \\ &+  \int_t^T (\rho_{\infty}\,\varphi)\ast \Delta K(s,X_{t,s}(x)) \diff s - \int_t^T  (\rho_{\infty}\,\varphi)\ast\nabla K \cdot \nabla V(s,X_{t,s}(x)) \diff s\\
&- \int_t^T g(s,X_{t,s}(x))\,(1+V(X_{t,s}(x))) \diff s,
\end{split}
\end{equation}
where $X_{t,s}(x)$ is the flow of the vector field $-K \ast (\nabla \mu_s + \mu_s \, \nabla V)$:
\begin{equation}\label{eq:ODE_for_the_flow}
\partial_s X_{t,s}(x) = - K \ast (\nabla \mu_s + \mu_s \, \nabla V)(X_{t,s}(x)), \qquad  X_{t,t}(x) = x.
\end{equation}

According to \eqref{eq:duality_estimate_on_measure_solution}, we need to estimate $\frac{\varphi(0,x)}{1+V(x)}$ uniformly with respect to ${g}$. From \eqref{eq:sol_dual_PDE} we have
\begin{align}\label{eq:sol_dual_PDE_case_Rd}
\frac{\varphi(0,x)}{1+V(x)} = & -\int_0^T \frac{(\nabla \rho_{\infty}\, \varphi) \ast \nabla K(s,X_{0,s}(x))}{1+V(x)} \diff s + \int_0^T (\nabla \rho_{\infty} \, \varphi)\ast K \cdot \frac{\nabla V(s,X_{0,s}(x))}{1+V(x)} \diff s \nonumber\\ &+  \int_0^T \frac{(\rho_{\infty}\,\varphi)\ast \Delta K(s,X_{0,s}(x))}{1+V(x)} \diff s - \int_0^T  (\rho_{\infty}\,\varphi)\ast\nabla K \cdot \frac{\nabla V(s,X_{0,s}(x))}{1+V(x)} \diff s\nonumber\\
&+ \int_0^T g(s,X_{0,s}(x))\, \frac{1+V(X_{0,s}(x))}{1+V(x)} \diff s,
\end{align}
First, we will need a lemma on quantities appearing in \eqref{eq:sol_dual_PDE_case_Rd}.
\begin{lem}\label{lem:quotients_v} Let $\{\mu_s\}_{s\in[0,T]}$ be the family of measures and let $X_{t,s}$ be defined by \eqref{eq:ODE_for_the_flow}. Then, there exists a constant $C$ depending only on $K$ and $V$ such that
$$
\left|\nabla X_{0,s}(x) \right| , \, \left|\frac{\nabla V(X_{0,s}(x))}{ 1+ V(x)}\right|, \, \left| \frac{V(X_{0,s}(x))}{1+ V(x)} \right|, \, \left|\frac{ \nabla^2 V(X_{0,s}(x))}{1+ V(x)} \right| \leq C\, e^{C\,\int_0^s \|\mu_u\|_{\BL^*_V} \diff u}. 
$$
\end{lem}
\begin{proof}
Note that
\begin{equation}\label{eq:ODE_explicitly}
X_{0,s}(x) = x - \int_0^s K \ast (\nabla \mu_u + \mu_u \, \nabla V)(X_{0,u}(x)) \diff u
\end{equation}
so that in particular
$$
\nabla X_{0,s}(x) = \mathbb{I}_d - \int_0^s \nabla K \ast (\nabla \mu_u + \mu_u \, \nabla V)(X_{0,u}(x)) \cdot \nabla X_{0,u}(x) \diff u,
$$
where $\mathbb{I}_d$ is the identity matrix. As $\left\|\nabla K \ast (\nabla \mu_u + \mu_u \, \nabla V)\right\|_{L^{\infty}(\R^d)} \leq C\, \|\mu_u\|_{\BL^*_V}$ (Lemma \ref{lem:est_vf_Rd}), the estimate on $\nabla X_{0,s}$ follows by Gr{\"o}nwall lemma.\\

We now proceed to the proof of the estimates involving potential $V$. We notice that the second term in \eqref{eq:ODE_explicitly} can be estimated by $\int_0^s \|\mu_u\|_{\BL^*_V} \diff u$ (Lemma \ref{lem:est_vf_Rd}). Now let $f$ be one of the functions $V$, $|\nabla V|$, $|\nabla^2 V|$ so that the target is to estimate $\frac{f(X_{0,s}(x))}{1+V(x)}$. We want to use the growth conditions \eqref{eq:growth_V_simple}. If $f$ happens to be bounded, the proof is concluded immediately. Otherwise, there exists $q \in \{p, p-1, p-2\}$, $q \geq 0$ such that
$$
|f(X_{0,s}(x))| \leq C\,(1+ |X_{0,s}(x)|^q) \leq C\, \left(1+ |x|^q + \left|\int_0^s K \ast (\nabla \mu_u + \mu_u \, \nabla V)(X_{0,u}(x)) \diff u \right|^q \right).
$$
Using Lemma \ref{lem:est_vf_Rd} and simple inequality $|x|^q \leq C\, e^{C\,|x|}$ we get
\begin{multline*}
\left|\int_0^s K \ast (\nabla \mu_u + \mu_u \, \nabla V)(X_{0,u}(x)) \diff u \right|^q \\ 
\leq \left|\int_0^s \left\| K \ast (\nabla \mu_u + \mu_u \, \nabla V) \right\|_{L^{\infty}(\R^d)} \diff u  \right|^q 
\leq C\, e^{C\,\int_0^s \|\mu_u\|_{\BL^*_V} \diff u},
\end{multline*}
It follows that
$$
\left|\frac{f(X_{0,s}(x))}{1+V(x)}\right|  \leq C+C\,\frac{|x|^q}{1+V(x)}  + C\, \frac{e^{C\,\int_0^s \|\mu_u\|_{\BL^*_V} \diff u}}{1+V(x)} \leq C \frac{|x|^q}{1+V(x)} + C\, e^{C\,\int_0^s \|\mu_u\|_{\BL^*_V} \diff u}.
$$
To conclude the proof, it remains to observe that because $0 \leq q \leq p$, the term $\frac{|x|^q}{1+V(x)}$ is bounded due to the growth conditions \eqref{eq:growth_V_simple}. 
\end{proof}
We proceed to the estimates on $\varphi$.
\begin{lem}\label{lem:Linf_weighted_dual_problem}
Let $\varphi$ be a solution to \eqref{eq:dual_PDE} with $g$ and $T>0$ fixed. Then, there exists a constant $C$ depending only on $V$ and $K$ such that
    $$
\left\|\frac{\varphi(0,\cdot)}{1+V(\cdot)}\right\|_{L^{\infty}(\R^d)} \leq C\,(T+1)\, \|g\|_{L^1(0,T; L^{\infty}(\R^d))}\, e^{C\,\int_0^T \|\mu_u\|_{\BL^*_V} \diff u}.
    $$
\end{lem}
\begin{proof} Using formula \eqref{eq:sol_dual_PDE_case_Rd} and estimating $1 \leq 1 + V(x)$, we have
\begin{equation*}
\begin{split}
\left\|\frac{\varphi(0,\cdot)}{1+V(\cdot)}\right\|_{L^{\infty}(\R^d)} \leq & \int_0^T \|(\nabla \rho_{\infty}\, \varphi) \ast \nabla K(s,\cdot)\|_{L^{\infty}(\R^d)}  \diff s \\  &+ \int_0^T \|(\nabla \rho_{\infty} \, \varphi)\ast K(s,\cdot) \|_{L^{\infty}(\R^d)} \, \left\|\frac{\nabla V(s,X_{0,s}(\cdot))}{1+V(\cdot)}\right\|_{L^{\infty}(\R^d)} \diff s \\ &+  \int_0^T \|{(\rho_{\infty}\,\varphi)\ast \Delta K(s,\cdot)} \|_{L^{\infty}(\R^d)} \diff s \\ 
&+ \int_0^T  \|(\rho_{\infty}\,\varphi)\ast\nabla K(s,\cdot)\|_{L^{\infty}(\R^d)} \, \left\|\frac{\nabla V(s,X_{0,s}(\cdot))}{1+V(\cdot)}\right\|_{L^{\infty}(\R^d)} \diff s\\
&+ \int_0^T \left\|g(s,\cdot)\right\|_{L^{\infty}(\R^d)} \, \left\|\frac{1+V(X_{0,s}(\cdot))}{1+V(\cdot)}\right\|_{L^{\infty}(\R^d)} \diff s.
\end{split}
\end{equation*}
Note that
\begin{equation}\label{eq:Linf_est_aux_est1}
\|\rho_{\infty}(\cdot)\,\varphi(s,\cdot)\|_{L^{1}(\Rd)} \leq \|\rho_{\infty}\|_{L^{1}(\Rd)}^{1/2}\, \mathcal{Q}(\varphi(s,\cdot)) \leq C\, \mathcal{Q}(\varphi(s,\cdot)),
\end{equation}
\begin{equation}\label{eq:Linf_est_aux_est2}
\|\nabla \rho_{\infty}(\cdot)\,\varphi(s,\cdot)\|_{L^{1}(\Rd)} \leq \|\rho_{\infty}\, |\nabla V|^2 \|_{L^{1}(\Rd)}^{1/2}\, \mathcal{Q}(\varphi(s,\cdot)) \leq  C\, \mathcal{Q}(\varphi(s,\cdot)),
\end{equation}
so that due to Theorem \ref{thm:main1} we have
\begin{align*}
\max(\|\rho_{\infty}(\cdot)\,\varphi(s,\cdot)\|_{L^{1}(\Rd)}, \|\nabla \rho_{\infty}(\cdot)\,\varphi(s,\cdot)\|_{L^{1}(\Rd)} ) & \leq C\, \mathcal{Q}(\varphi(s,\cdot)) \\ &\leq C\,\|g\|_{L^1(0,T; L^{\infty}(\Rd))} \, e^{C\,\int_0^T \|\mu_u\|_{\BL^*_V} \diff u}.
\end{align*}
By Young's convolutional inequality, for any function $f:\R^d \to \R$,
\begin{equation}\label{eq:estimates_products_rho_phi}
\begin{split}
\|(\rho_{\infty}\,\varphi)\ast f(s,\cdot)\|_{L^{\infty}(\Rd)}, \|(\nabla \rho_{\infty}\,\varphi)&\ast f(s,\cdot)\|_{L^{\infty}(\Rd)} \leq \\ &\leq
C\, \|f\|_{L^{\infty}(\Rd)}\,\|g\|_{L^1(0,T; L^{\infty}(\Rd))} \, e^{C\,\int_0^T \|\mu_u\|_{\BL^*_V} \diff u}.
\end{split} 
\end{equation} 
Hence, applying it with $f = K, \Delta K, \partial_{x_i} K$ (for all $i = 1, ..., d$) and using Lemma \ref{lem:quotients_v} for the potential term, we conclude the proof.
\end{proof} 

\begin{lem}\label{lem:varphi_estimate_grad}
Let $\varphi$ be a solution to \eqref{eq:dual_PDE} with $g$ and $T>0$ fixed. Then, there exists a constant $C$ depending only on $V$ and $K$ such that
\begin{equation}\label{eq:estimate_grad_dual_sol}
\left\|\nabla \frac{\varphi(0,\cdot)}{1+V(\cdot)}\right\|_{L^{\infty}(\R^d)} \leq C\, (T+1)\, \|g\|_{L^1(0,T; \BL(\R^d))}\, e^{C\,\int_0^T \|\mu_u\|_{\BL^*_V} \diff u}
\end{equation}    
\end{lem}
\begin{proof}
Note that
$$
\nabla \frac{\varphi(0,x)}{1+V(x)} = \frac{\nabla \varphi(0,x)}{1+V(x)} - \frac{\varphi(0,x)}{1+V(x)}\, \frac{\nabla V(x)}{1+V(x)}.    
$$
The second term is bounded by Lemma \ref{lem:Linf_weighted_dual_problem} and the assumption on the potential so it is sufficient to estimate $\frac{\nabla \varphi(0,x)}{1+V(x)}$. Differentiating each term in \eqref{eq:sol_dual_PDE} with respect to $x$ at $t=0$, dividing by $(1+V(x))$ and estimating $1\leq 1+V(x)$ when there are no terms with the potential $V$ in the numerator, we get
\begin{equation*}\label{eq:sol_dual_PDE_der_weighted}
\begin{split}
&\left\|\frac{\nabla \varphi(0,\cdot)}{1+V(\cdot)}\right\|_{L^{\infty}(\R^d)} \leq   \int_0^T \left\|(\nabla \rho_{\infty}\, \varphi) \ast \nabla^2 K(s,\cdot)\right\|_{L^{\infty}(\R^d)}\, \left\| \nabla X_{0,s}\right\|_{L^{\infty}(\R^d)}   \diff s \\ 
& \qquad + \int_0^T \left\|(\nabla \rho_{\infty} \, \varphi)\ast \nabla K(s,\cdot)\right\|_{L^{\infty}(\R^d)}\,  \left\| \frac{\nabla V(s,X_{0,s}(\cdot))}{1+V(\cdot)}\right\|_{L^{\infty}(\R^d)}  \, \left\| \nabla X_{0,s}\right\|_{L^{\infty}(\R^d)}  \diff s
\\ 
& \qquad + \int_0^T \left\|(\nabla \rho_{\infty} \, \varphi)\ast  K(s,\cdot)\right\|_{L^{\infty}(\R^d)}\,  \left\| \frac{\nabla^2 V(s,X_{0,s}(\cdot))}{1+V(\cdot)}\right\|_{L^{\infty}(\R^d)}  \, \left\| \nabla X_{0,s}\right\|_{L^{\infty}(\R^d)}  \diff s\
\\
& \qquad +  \int_0^T \left\|(\rho_{\infty}\,\varphi)\ast \nabla \Delta K(s,\cdot) \right\|_{L^{\infty}(\R^d)} \, \left\| \nabla X_{0,s}\right\|_{L^{\infty}(\R^d)} \diff s\\ 
\end{split}
\end{equation*}
\begin{equation*}\label{eq:sol_dual_PDE_der_weighted}
\begin{split}
& \qquad +\int_0^T  \left\|(\rho_{\infty}\,\varphi)\ast\nabla^2 K (s,\cdot) \right\|_{L^{\infty}(\R^d)} \, \left\| \frac{\nabla V(s,X_{0,s}(\cdot))}{1+V(\cdot)} \right\|_{L^{\infty}(\R^d)}\, \left\| \nabla X_{0,s}\right\|_{L^{\infty}(\R^d)} \diff s\\
& \qquad +\int_0^T  \left\|(\rho_{\infty}\,\varphi)\ast\nabla K (s,\cdot) \right\|_{L^{\infty}(\R^d)} \, \left\| \frac{\nabla^2 V(s,X_{0,s}(\cdot))}{1+V(\cdot)} \right\|_{L^{\infty}(\R^d)}\, \left\| \nabla X_{0,s}\right\|_{L^{\infty}(\R^d)} \diff s\\
& \qquad + \int_0^T \left\|\nabla g(s,\cdot) \right\|_{L^{\infty}(\R^d)} \, \left\| \frac{1+ V(s,X_{0,s}(\cdot))}{1+V(\cdot)} \right\|_{L^{\infty}(\R^d)} \left\| \nabla X_{0,s}\right\|_{L^{\infty}(\R^d)} \,  \diff s\\
& \qquad + \int_0^T \left\|g(s,\cdot) \right\|_{L^{\infty}(\R^d)} \, \left\| \frac{\nabla V(s,X_{0,s}(\cdot))}{1+V(\cdot)} \right\|_{L^{\infty}(\R^d)} \left\| \nabla X_{0,s}\right\|_{L^{\infty}(\R^d)} \,  \diff s.
\end{split}
\end{equation*}
Now, we obtain \eqref{eq:estimate_grad_dual_sol} directly from Lemma \ref{lem:quotients_v}, estimates \eqref{eq:estimates_products_rho_phi} and the fact that $\left\|\nabla g(s,\cdot) \right\|_{L^{\infty}(\R^d)} \leq \left|g(s,\cdot)\right|_{\Lip}$. 
\end{proof}

\begin{proof}[Proof of Theorem \ref{thm:main1}]
Thanks to Lemmas \ref{lem:Linf_weighted_dual_problem} and \ref{lem:varphi_estimate_grad} we know that
$$
\left\| \frac{\varphi(0,\cdot)}{1+V(\cdot)} \right\|_{\BL(\R^d)} \leq C\,(T+1)\, \|g\|_{L^1(0,T; \BL(\R^d))}\, e^{C\,\int_0^T \|\mu_u\|_{\BL^*_V} \diff u}.
$$
where the constant $C$ does not depend on $g$ and $T$. Using duality formula \eqref{eq:duality_estimate_on_measure_solution} and taking supremum over all $g$ such that $\| g\|_{L^1(0,T; \BL(\R^d))} \leq 1$ we obtain
$$
\esssup_{t \in [0,T]} \| \mu_t \|_{\BL^*_V} \leq C\,(T+1)\,  e^{C\,\int_0^T \|\mu_u\|_{\BL^*_V} \diff u}\, \| \mu_0 \|_{\BL^*_V}.
$$
By continuity of the map $[0,T] \ni u \mapsto \mu_u$ in the $\|\cdot\|_{\BL^*_V}$ norm (Lemma \ref{lem:continuity_in_time}), we conclude 
$$
\| \mu_T \|_{\BL^*_V} \leq C\,(T+1)\,  e^{C\,\int_0^T \|\mu_u\|_{\BL^*_V} \diff u}\,  \| \mu_0 \|_{\BL^*_V}.
$$
Using Lemma \ref{lem:gronwall_exponential_ricatti}, we arrive at \eqref{eq:stability_estimate_main_thm}.
\end{proof}

\appendix 
\section{Gr{\"o}nwall-type inequalities}

\begin{lem}[backward Gr{\"o}nwall's inequality]\label{lem:back_gronwall}
    Suppose that $f, g,h: [0,T]\to \R^+$ such that $h$ is nonincreasing,  $C$ is a nonnegative constant and 
    $$
    f(t) \leq h(t) + C\, \int_t^T g(s) \, f(s) \diff s.
    $$
    Then, $f(t) \leq h(t)\,e^{C\,\int_t^{T} g(u) \diff u}$.
\end{lem}
\begin{proof}
We change variables $u = T-s$ so that
$$
f(T-(T-t)) \leq h(T-(T-t)) + C \, \int_{0}^{T-t} g(T-u)\, f(T-u) \diff u.
$$
Applying usual Gr{\"o}nwall's inequality to the function $s \mapsto f(T-s)$ (note that the function $s \mapsto h(T-s)$ is nondecreasing) we deduce
$$
f(t) = f(T-(T-t)) \leq h(T-(T-t))\,e^{C\,\int_0^{T-t} g(T-u) \diff u } = h(t)\,e^{C\,\int_t^{T} g(u) \diff u }.
$$
\end{proof}

\begin{lem}\label{lem:gronwall_exponential_ricatti}
Let $y(t):[0,\infty)\to \R^+$ be a continuous function such that 
$$
y(t)\leq \alpha(t)\, e^{C\,\int_0^t y(s) \diff s}
$$
for some $C>0$ and nondecreasing, nonnegative function $\alpha(t)$. Then,
$$
y(t) \leq \frac{\alpha(t)}{1-C\,t\,\alpha(t)}
$$
whenever $1-C\,t\,\alpha(t)>0$.
\end{lem}
\begin{proof}
We slightly adapt the proof from \cite[Lemma 3]{MR2448326}. We fix $T>0$ and consider $t\in[0,T]$. Then,
$$
y(t)\leq \alpha(T)\, e^{C\,\int_0^t y(s) \diff s}.
$$
We let $z(t) = \alpha(T)\, e^{C\,\int_0^t y(s) \diff s}$ and we note that 
$
z'(t) = C\, z(t)\, y(t) \leq C\, z(t)^2. 
$
Integrating this differential inequality, we get
$$
z(t) \leq \frac{z(0)}{1-t\,z(0)} = \frac{\alpha(T)}{1-C\,t\,\alpha(T)}.
$$
Taking $t=T$, we conclude the proof.
\end{proof}

\section{Continuity of solutions in time}\label{app:cont_in_time}
\begin{proof}[Proof of Lemma \ref{lem:continuity_in_time}]
From \cite{MR3919409}, we know that the measure solution is the fixed point of the push-forward representation
\begin{equation}\label{eq:push-forward_rho_represent}
\rho_t = X_{0,t}^{\#}\,\rho_0,
\end{equation}
where $X_{0,t}$ is the flow of the related vector field
\begin{equation}\label{eq:cont_sol_flow_ODE}
\partial_t X_{0,t}(x) = -K \ast (\nabla \rho_t + \rho_t \, \nabla V)(X_{0,t}), \qquad X_{0,0}(x) = x.
\end{equation}
Note carefully that since $\nabla \rho_{\infty} + \rho_{\infty}\,\nabla V = 0$, we have $(\nabla \rho_t + \rho_t \, \nabla V) = (\nabla \mu_t + \mu_t \, \nabla V)$ so that the flow map $X_{0,t}$ is exactly the one defined in \eqref{eq:ODE_for_the_flow}. This fact will be relevant in the sequel. \\

Given a test function $\psi \in \BL(\R^d)$ with $\|\psi\|_{\BL(\R^d)}\leq 1$, and times $s, t \in [0,T]$ we compute using \eqref{eq:push-forward_rho_represent} and \eqref{eq:cont_sol_flow_ODE}
\begin{align*}
\int_{\R^d} \psi(x)&\,(1+V(x))\diff (\rho_t - \rho_s)(x) = \\ &=
\int_{\R^d} \left[\psi(X_{0,t}(x))\,(1+V(X_{0,t}(x))) - \psi(X_{0,s}(x))\,(1+V(X_{0,s}(x)))  \right] \diff \rho_0(x)\\
&= \int_{\R^d} \left[\psi(X_{0,t}(x)) - \psi(X_{0,s}(x))\, \right] \, \frac{1+V(X_{0,t}(x))}{1+V(x)}\, (1+V(x)) \diff \rho_0(x) \\
&\phantom{ = } + 
\int_{\R^d} \psi(X_{0,s}(x)) \frac{V(X_{0,t}(x)) - V(X_{0,s}(x))}{1+V(x)}\, (1+V(x)) \diff \rho_0(x) =: I_1 + I_2 .
\end{align*}
For the term $I_1$ we use \eqref{eq:cont_sol_flow_ODE}, Lemma \ref{lem:est_vf_Rd} (to control the vector field) and \eqref{eq:estimate_C(T)_on_finite_time_intervals} (to control the solution on bounded intervals of time)
\begin{align*}
|X_{0,t}(x)-X_{0,s}(x)| &\leq \int_s^t \left\|K \ast (\nabla \rho_u + \rho_u \, \nabla V)\right\|_{L^{\infty}(\R^d)} \diff u \\ &\leq \int_s^t \|\rho_u\|_{\BL^*_V}\, \diff u \leq C(T)\,|t-s|.
\end{align*}
Moreover, by Lemma \ref{lem:quotients_v} and \eqref{eq:estimate_C(T)_on_finite_time_intervals}, 
\begin{equation}\label{eq:quotients_estimates_continuity_proof}
\left|\frac{1+V(X_{0,t}(x))}{1+V(x)}\right|, \left| \frac{\nabla V(X_{0,t}(x))}{1+V(x)} \right| \leq C\, e^{C\, \int_s^t \|\rho_u\|_{\BL^*_V} \diff u} \leq C(T).
\end{equation}
Hence, using 1-Lipschitz continuity of $\psi$ and 
\eqref{eq:quotients_estimates_continuity_proof} we obtain that  
$$
|I_1| \leq C(T)\, |t-s|\,  \int_{\R^d} (1+V(x))\diff \rho_0(x) \leq C(T,\rho_0)\,|t-s|.
$$
For $I_2$, we first estimate
\begin{align*}
\left|\frac{V(X_{0,t}(x)) - V(X_{0,s}(x))}{1+V(x)}\right| & \leq  \int_s^t \left|\frac{\nabla V(X_{0,u})}{1+V(x)}\right| \, \left|K \ast (\nabla \rho_u + \rho_u \, \nabla V)(X_{0,u})\right| \diff u \\ & \leq C(T)\, \int_s^t \|\rho_u\|_{\BL^*_V} \diff u \leq C(T)\, |t-s|, 
\end{align*}
where we used \eqref{eq:quotients_estimates_continuity_proof} and Lemma \ref{lem:est_vf_Rd}. Hence, since $\|\psi\|_{L^{\infty}(\R^d)} \leq 1$, we obtain
$$
|I_2| \leq C(T)\, |t-s|\,  \int_{\R^d} (1+V(x))\diff \rho_0(x) \leq C(T,\rho_0)\,|t-s|.
$$
It follows that 
$$
\left| \int_{\R^d} \psi(x)\,(1+V(x))\diff (\rho_t - \rho_s)(x) \right| \leq C(T,\rho_0)\,|t-s|.
$$
Taking supremum over all $\psi \in \BL(\R^d)$ with $\|\psi\|_{\BL(\R^d)}\leq 1$, we conclude the proof.
\end{proof}

\section{Technical proofs from Section \ref{sect:assumption}}\label{app:proofs_assumptions}

\begin{proof}[Proof of Lemma \ref{lem:example_ass_satisfied}]
Let $g(x,y) = \frac{\nabla V(x)\cdot \nabla V(y)}{1+V(y)}\, K(x-y)$ and $h(x,y) = \frac{\nabla V(x)}{1+V(y)}\, K(x-y)$. If $p > 2$, we see, taking $x=y$, that $g$ is not bounded which proves that $p \in (0,2]$ is a necessary condition.\\

Let $p \in (0,2]$. We need to prove that $g, \nabla_y g, h, \nabla_y h \in L^{\infty}(\R^d \times \R^d)$. We will use the following inequality
\begin{equation}\label{eq:triangle_inequality_nablaV}
|\nabla V(x)| \leq C\,|\nabla V(y)| + C\, |\nabla V(x-y)| + C
\end{equation} 
which is a consequence of \eqref{eq:growth_V_simple}. \\

\underline {Boundedness of $g$.} Using \eqref{eq:triangle_inequality_nablaV} we have
\begin{multline*}
\frac{1}{C} \, |g(x,y)| \leq  \left(\frac{|\nabla V(y)|\, |\nabla V(y)|}{1+V(y)} + \frac{|\nabla V(x-y)|\, |\nabla V(y)|}{1+V(y)} +  \frac{|\nabla V(y)|}{1+V(y)}\right) K(x-y). 
\end{multline*} 
The first and third terms are controlled since $p \leq 2$ while the second uses additionally the control of $|\nabla V|\, K$.\\

\underline{Boundedness of $h$.} We use \eqref{eq:triangle_inequality_nablaV} to get
$$
\frac{1}{C} |h(x,y)| \leq \frac{|\nabla V(y)|}{1+V(y)}\, K(x-y) + \frac{|\nabla V(x-y)|}{1+V(y)}\, K(x-y) + \frac{1}{1+V(y)}\, K(x-y).
$$
To conclude, we use boundedness of $\frac{\nabla V}{1+V}$ (which holds for any $p>0$) and $|\nabla V|\, K$.\\

\underline{Boundedness of $\nabla_y h$.} By a direct computation,
$$
\nabla_y h(x,y) = -\frac{\nabla V(x) \otimes \nabla V(y)}{(1+V(y))^2} \, K(x-y) - \frac{\nabla V(x)\otimes \nabla K(x-y)}{1+V(y)} =: R_1 + R_2.
$$
Using \eqref{eq:triangle_inequality_nablaV} we get
\begin{multline*}
\frac{\left| R_1 \right|}{C} \leq \left|\frac{\nabla V(y) \otimes \nabla V(y)}{(1+V(y))^2} \, K(x-y) \right| +\\ + \left|\frac{\nabla V(x-y) \otimes \nabla V(y)}{(1+V(y))^2} \, K(x-y)\right| +\frac{|\nabla V(y)|}{(1+V(y))^2} \, K(x-y),
\end{multline*}
$$
\frac{\left| R_2 \right|}{C} \leq \left|\frac{\nabla V(y)\otimes \nabla K(x-y)}{1+V(y)}\right|  +  \left|\frac{\nabla V(x-y)\otimes \nabla K(x-y)}{1+V(y)}\right| + \frac{|\nabla K(x-y)|}{1+V(y)} . 
$$
All the terms above are bounded because $\frac{\nabla V}{1+V}$, $|\nabla V| \,K$ and $|\nabla V| \,|\nabla K|$ are bounded.\\

\underline {Boundedness of $\nabla_y g$.} By a direct computation
\begin{multline*}
\nabla_y g(x,y) = \frac{\nabla V(x)\cdot \nabla^2 V(y)}{1+V(y)}\, K(x-y)- \frac{\nabla V(x)\cdot \nabla V(y)\, \nabla V(y)}{(1+V(y))^2}\, K(x-y)  \\ -
\frac{\nabla V(x)\cdot \nabla V(y)}{1+V(y)}\, \nabla K(x-y) =: P_1 + P_2 + P_3.
\end{multline*}
Concerning the term $P_1$, we notice that since $p \leq 2$, $|\nabla^2 V| \leq C$ so that $P_1$ can be estimated by $\frac{|\nabla V(x)|}{1+V(y)}\, K(x-y) = |h(x,y)|$ which was proved to be bounded above.

Concerning the term $P_2$, we use \eqref{eq:triangle_inequality_nablaV} to get
$$
\frac{|P_2|}{C}\leq \|K\|_{L^{\infty}(\R^d)}\, \left\|\frac{|\nabla V|^3}{(1+V)^2}\right\|_{L^{\infty}(\R^d)}
+ \left\| \frac{|\nabla V|^2}{(1+V)^2}\right\|_{L^{\infty}(\R^d)} \, \left( \left\| \nabla V\, K \right\|_{L^{\infty}(\R^d)} + \left\|  K \right\|_{L^{\infty}(\R^d)} \right).
$$
By the growth conditions \eqref{eq:growth_V_simple} and $p\leq 2$, $\left\|\frac{|\nabla V|^3}{(1+V)^2}\right\|_{L^{\infty}(\R^d)}$ is finite and so, $P_2$ is bounded.

Concerning the term $P_3$, we argue as in $P_2$ to get
$$
\frac{|P_3|}{C} \leq \|\nabla K\|_{L^{\infty}(\R^d)}\, \left\|\frac{|\nabla V|^2}{1+V}\right\|_{L^{\infty}(\R^d)}
+ \left\| \frac{|\nabla V|}{1+V}\right\|_{L^{\infty}(\R^d)} \, \left( \left\| \nabla V\, \nabla K \right\|_{L^{\infty}(\R^d)} + \left\| \nabla K \right\|_{L^{\infty}(\R^d)}\right).
$$
The term $\left\|\frac{|\nabla V|^2}{1+V}\right\|_{L^{\infty}(\R^d)}$ is bounded because $p \leq 2$ and all the other terms are bounded by assumption. The proof is concluded.
\end{proof}

\begin{proof}[Proof of Lemma \ref{lem:est_vf_Rd}]
Concerning \eqref{eq:estimate_vf_whole_space_1}, we only prove the first estimate. The second can be proved in the same way, replacing $K$ with $\nabla K$. We need to study two terms $\nabla K \ast \mu$ and $K \ast (\mu\,\nabla V)$. For the first one,
\begin{multline*}
\|\nabla K \ast \mu \|_{L^{\infty}(\R^d)} = \sup_{x \in \R^d} \left|\int_{\R^d} \frac{\nabla K(x-y)}{1+V(y)}\, (1+V(y))  \diff \mu(y)\right| \leq \\
\leq 
\sup_{x\in\R^d} \left\| \frac{\nabla K(x-\cdot)}{1+V(\cdot)} \right\|_{\BL} \, \| \mu\|_{\BL^*_V} \leq \sup_{x\in\R^d} \|\nabla K(x-\cdot)\|_{\BL} \, \left\| \frac{1}{1+V} \right\|_{\BL} \, \| \mu\|_{\BL^*_V},
\end{multline*}
where $ \frac{1}{1+V} \in \BL(\R^d)$ thanks to the growth condition \eqref{eq:growth_V_simple}. For the second one, we write  
\begin{align*}
\|K \ast (\mu\,\nabla V)& \|_{L^{\infty}(\R^d)} = 
\sup_{x\in\R^d} \left| \int_{\R^d} K(x-y) \, \frac{\nabla V(y)}{1+V(y)}\, (1+V(y)) \diff \mu(y)\right| \leq \\ &
\leq 
\sup_{x\in\R^d} \left\| K(x-\cdot) \, \frac{\nabla V(\cdot)}{1+V(\cdot)} \right\|_{\BL}
\leq 
\sup_{x\in\R^d} \|K(x-\cdot)\|_{\BL}\, \left\| \frac{\nabla V}{1+V} \right\|_{\BL} \, \| \mu \|_{\BL^*_V} \\
&\leq \sup_{x\in\R^d} \|K\|_{\BL}\, \left\| \frac{\nabla V}{1+V} \right\|_{\BL} \, \| \mu \|_{\BL^*_V},
\end{align*}
where $\frac{\nabla V}{1+V} \in \BL(\R^d)$ due to the growth condition \eqref{eq:growth_V_simple}. We proceed to the proof of \eqref{eq:estimate_vf_whole_space_2} which requires condition \eqref{eq:add_cond_growth_at_infinity}. As before, we write 
\begin{multline*}
\left|\nabla V(x) \cdot K \ast \nabla \mu(x)\right| = \left|\int_{\R^d} \nabla V(x)\cdot \nabla K(x-y)\,  \diff \mu(y) \right| \leq \\ \leq \sup_{x\in\R^d} \left\|  \frac{\nabla V(x)}{1+V(\cdot)}\cdot \nabla K(x-\cdot)\right\|_{\BL} \, \|\mu\|_{\BL^*_V}.  
\end{multline*}
Finally,  we conclude
\begin{multline*}
\left| \nabla V(x) \cdot K \ast (\mu \, \nabla V)(x) \right| = \left| \int_{\R^d} \nabla V(x) \cdot \nabla V(y)\, K(x-y)\diff \mu(y) \right| \leq \\
\leq \sup_{x\in\R^d} \left\| \frac{\nabla V(x) \cdot \nabla V(\cdot)}{1+V(\cdot)}\, K(x-\cdot) \right\|_{\BL} \, \|\mu \|_{\BL^*_V}.  
\end{multline*}
\end{proof}

\subsection*{Acknowledgements}
JAC and JS were supported by the Advanced Grant Nonlocal-CPD (Nonlocal PDEs for Complex Particle Dynamics: Phase Transitions, Patterns and Synchronization) of the European Research Council Executive Agency (ERC) under the European Union’s Horizon 2020 research and innovation programme (grant agreement No. 883363). JAC was also partially supported by the EPSRC grant numbers EP/T022132/1 and EP/V051121/1.

\bibliographystyle{abbrv}
\bibliography{fastlimit}
\end{document}